\documentclass[a4paper,11pt]{amsart}

\usepackage{amsmath,amssymb,amsthm,enumitem,hyperref,tikz,graphicx,mathtools,comment}
\usepackage[left=2.5cm,right=2.5cm,top=2.6cm,bottom=2.6cm]{geometry}
\setenumerate{label=(\roman*)}
\usetikzlibrary{cd}

\DeclareMathOperator{\Aut}{Aut}
\DeclareMathOperator{\Out}{Out}

\DeclareMathOperator{\Id}{Id}
\DeclareMathOperator{\Cay}{Cay}
\DeclareMathOperator{\Hom}{Hom}
\DeclareMathOperator{\Stab}{Stab}
\newcommand{\ad}{\mathrm{ad}}
\DeclareMathOperator{\Inn}{Inn}
\DeclareMathOperator{\GL}{GL}
\newcommand{\ZZ}{\mathbb{Z}}
\newcommand{\NN}{\mathbb{N}}

\newcommand{\IA}{\operatorname{IA}_n(\mathbb{Z} / 3\mathbb{Z})}
\newcommand{\FF}[1]{\mathbb{F}^{(#1)}}
\newcommand{\FFF}[1]{\widehat{\mathbb{F}}^{(#1)}}
\newcommand{\FFFF}[1]{\widetilde{\mathbb{F}}^{(#1)}}

\theoremstyle{definition}
\newtheorem{defn}{Definition}[section]
\theoremstyle{plain}
\newtheorem{thm}[defn]{Theorem}
\newtheorem{lem}[defn]{Lemma}
\newtheorem{prop}[defn]{Proposition}
\newtheorem{cor}[defn]{Corollary}
\theoremstyle{remark}
\newtheorem{remark}[defn]{Remark}

\title{Free-by-cyclic groups are equationally Noetherian}
\author{Monika Kudlinska}
\address[M.\ Kudlinska]{Emmanuel College, University of Cambridge, St Andrew's Street, Cambridge CB2~3AP, United Kingdom}
\email{mak74@cam.ac.uk}
\author{Motiejus Valiunas}
\address[M.\ Valiunas]{Instytut Matematyczny, Uniwersytet Wroc{\l}awski, plac Grunwaldzki 2, 50-384 Wroc{\l}aw, Poland}
\email{motiejus.valiunas@math.uni.wroc.pl}


\begin{document}

\begin{abstract}
A group $G$ is said to be equationally Noetherian if every system of equations in $G$ is equivalent to a finite subsystem. We show that all free-by-cyclic groups are equationally Noetherian. As a corollary, we deduce that the set of exponential growth rates of a free-by-cyclic group is well ordered. Along the way, we prove that free-by-cyclic groups with polynomially growing monodromies of infinite order admit non-elementary 4-acylindrical actions on trees. We show that the splittings arising from the improved relative train track machinery of Bestvina--Feighn--Handel \cite{BestvinaFeighnHandel2005} are 2-acylindrical, when the growth is at least quadratic.
\end{abstract}
\maketitle

\section{Introduction}

A group $G$ is \emph{free-by-cyclic} if it admits the structure of a semidirect product $G = F \rtimes_{\varphi} \ZZ$ where $F$ is a finitely generated free group. It is easy to check that the group $G$ only depends on the outer automorphism $\Phi \in \Out(F)$ represented by $\varphi$, and so we may write $F \rtimes_\Phi \ZZ$ for $G$. We will sometimes call $G$ the \emph{mapping torus} of $\varphi$ or $\Phi$, and $\Phi$ the \emph{monodromy} of $G$. There is a close link between the algebraic and geometric properties of a free-by-cyclic group $F \rtimes_{\Phi} \ZZ$, and the dynamics of the monodromy map $\Phi \in \Out(F)$ \cite{Brinkmann2000, DahmaniLi2022, Hagen2019}.

An important open question is whether all free-by-cyclic groups are linear. By the work of Hagen--Wise \cite{HagenWise2015}, free-by-cyclic groups which are hyperbolic admit proper and cocompact actions on CAT(0) cube complexes. Hence, the work of Agol \cite{Agol2013} implies that hyperbolic free-by-cyclic groups are virtually special and thus linear. On the other hand, Button shows that there exist free-by-cyclic groups which fail to be linear over any field of positive characteristic \cite[Corollary~4.7]{Button2017}.

In this paper we study a generalisation of linearity for groups. A group $G$ is said to be \emph{equationally Noetherian} if any system of equations over $G$ is equivalent to a finite subsystem. This notion, originating in the context of Algebraic Geometry over Groups \cite{BaumslagMyasnikovRemeslennikov1999}, has been thereafter studied from a geometric point of view, and can be interpreted as the study of the set $\Hom(H,G)$ for a finitely generated group $H$ \cite{Sela2001}. Notably, such geometric methods have been used to show that groups hyperbolic relative to equationally Noetherian groups are equationally Noetherian themselves \cite{GrovesHull2019}. 

Our main result is the following.

\begin{thm} \label{thm:main}
All free-by-cyclic groups are equationally Noetherian.
\end{thm}

Let us now briefly outline the proof of Theorem~\ref{thm:main}; we refer the reader to Section~\ref{s:autos} for the definitions of growth of automorphisms and unipotent automorphisms. 

We begin by considering the case where the free-by-cyclic group $G = F \rtimes_{\Phi} \ZZ$ has  unipotent and polynomially growing monodromy. We proceed by induction on the polynomial growth rate $d$ of $\Phi$. When $d > 0$, it is a well-known fact that the free-by-cyclic group $G$ acts on a simplicial tree, where the vertex stabilisers are free-by-cyclic with polynomially growing monodromies of growth rate strictly less than $d$ \cite[Proposition~2.5]{AndrewHughesKudlinska2023}. When the growth rate is at least quadratic, it is possible to find such a tree $\mathcal{T}$ where the edge stabilisers are infinite cyclic. We show that the action of the free-by-cyclic group on $\mathcal{T}$ is 4-acylindrical and use the criterion of the second author in \cite[Thereom~1.9]{Valiunas2021} to conclude that $G$ is equationally Noetherian. 

In the case that the growth of $\Phi$ is linear, by the work of Andrew--Martino \cite{AndrewMartino2022} there is an action of $G$ on a tree with free abelian edge stabilisers and vertex stabilisers of the form $F_v \times \ZZ$ where $F_v$ is free. We show that this action is also 4-acylindrical. Moreover, we give a new criterion (Proposition~\ref{prop:cyclic-ext}) which ensures that such a splitting gives rise to an equationally Noetherian group.  A key ingredient in Proposition~\ref{prop:cyclic-ext}, as well as Theorem~1.9 in \cite{Valiunas2021}, is the machinery of \emph{non-divergent sequences} due to Groves--Hull \cite{GrovesHull2019}.

For the general case, we note that any free-by-cyclic group with exponentially growing monodromy is hyperbolic relative to a collection of free-by-cyclic subgroups with polynomially growing monodromies \cite{DahmaniLi2022}. This, combined with the work of Groves--Hull \cite[Theorem~D]{GrovesHull2019}, implies Theorem~\ref{thm:main} assuming the polynomially growing monodromy case.

\medskip

The main application of our result is to the study of \emph{growth rates} of free-by-cyclic groups. Let $G$ be a group and $S$ a finite generating set of $G$. The \emph{exponential growth rate} of $(G, S)$ is defined to be \[e(G, S) \coloneqq \lim_{n\to \infty} |B_n(G, S)|^{\frac{1}{n}}\]
where $B_n(G, S)$ denotes the ball of radius $n$ around the identity in the Cayley graph of $G$ with respect to the generating set $S$, and $|B_n(G, S)|$ denotes the number of vertices in $B_n(G, S)$. A group $G$ is said to be \emph{exponentially growing} if there exists a finite generating set $S$ such that $e(G, S) > 1$. In that case, it is interesting to study the set of all exponential growth rates of $G$, 
\[\xi(G) \coloneqq \{e(G, S) \mid S \text{ finite generating set of }G\}.\]
A group $G$ is said to admit \emph{uniform exponential growth} if $\inf \xi(G) > 1$. For any group $G$ which is free-by-cyclic and not virtually abelian, $G$ is known to have uniform exponential growth by \cite[Lemma~2.3]{CeccheriniGrigorchuk1997}. We give a wide-reaching generalisation of this result by showing the following.

\begin{cor}\label{cor:growth-rates-order}
    If $G$ is free-by-cyclic then the set of exponential growth rates $\xi(G)$ of $G$ is well ordered.
\end{cor}

Fujiwara--Sela \cite{FujiwaraSela2023} prove that the set of exponential growth rates of a Gromov hyperbolic group is well ordered. Their theorem is an algebraic analogue of a celebrated result of J{\o}rgensen and Thurston which states that the set of volumes of finite volume complete hyperbolic 3-manifolds is well ordered \cite{Thurston}.

The work of Fujiwara--Sela has been extended by Fujiwara to certain acylindrically hyperbolic groups in \cite{Fujiwara2023}. More precisely, Fujiwara shows that if an equationally Noetherian group $G$ admits a non-elementary acylindrical action on a hyperbolic space $X$, and there exists some positive integer $M$ such that for any finite generating set $S$ satisfying $1 \in S = S^{-1}$, the set $S^M$ contains a hyperbolic element for the action $G \curvearrowright X$, then the set of exponential growth rates of $G$ is well ordered \cite[Theorem~1.1]{Fujiwara2023}. Our Corollary~\ref{cor:growth-rates-order} follows by applying this result to the setting of free-by-cyclic groups, combining our Theorem~\ref{thm:main} with an observation that all free-by-cyclic groups with polynomially growing monodromies admit 4-acylindrical actions on simplicial trees.

A consequence of Corollary~\ref{cor:growth-rates-order} is that a free-by-cyclic group $G$ (of exponential growth) has a minimal exponential growth rate, i.e.\ $e(G,S) = \inf \xi(G)$ for some finite generating set $S$ of $G$. This minimal growth rate is a natural candidate for an algebraic notion of volume in groups. Other notions of volume for free-by-cyclic groups include the \emph{$\ell^2$-torsion} (see \cite[Chapter~3]{Lueck2002}) and \emph{minimal volume entropy} defined by Bregman--Clay in \cite{BregmanClay2021}. It is interesting to note that whilst the $\ell^2$-torsion $\rho^{(2)}$ vanishes for every free-by-cyclic group with polynomially growing monodromy \cite[Theorem~5.1]{Clay2017}, and the minimal volume entropy $\mathcal{E}_{\rm min}$ vanishes for certain free-by-cyclic groups with linearly-growing monodromies \cite[Theorem~1.1]{BregmanClay2021}, the logarithm of the minimal exponential growth rate of any free-by-cyclic group of exponential growth is positive.

\medskip

The structure of the paper is as follows.  In Section~\ref{s:background}, we cover the preliminary notions on graphs of groups, equationally Noetherian groups and automorphisms of free groups. In Section~\ref{s:acylindrical}, we show that certain actions of free-by-cyclic groups on trees are acylindrical. We prove Theorem~\ref{thm:main} by induction on the growth of the automorphism $\Phi \in \Out(F)$ defining a free-by-cyclic group $G = F \rtimes_\Phi \ZZ$: we prove the base case (when $\Phi$ has linear growth) in Section~\ref{s:linear}, and the general case in Section~\ref{s:superlinear}.  We prove Corollary~\ref{cor:growth-rates-order} in Section~\ref{s:growth-rates}.

\subsection*{Acknowledgements}

We would like to thank Naomi Andrew, Daniel Groves, Alice Kerr, Armando Martino and the anonymous referee for valuable comments. We would also like to thank Naomi Andrew and Camille Horbez for a number of useful discussions. The second author was partially supported by the National Science Centre (Poland) grant No.\  2022/47/D/ST1/00779.

\section{Background}
\label{s:background}

\subsection{Graphs and graphs of groups}

Throughout the paper, graphs are undirected (although each edge is sometimes viewed as a pair of oriented edges), possibly with loops or multiple edges. Given a graph $\Gamma$, we write $V(\Gamma)$ and $E(\Gamma)$ for the set of vertices and directed edges of $\Gamma$, respectively. For any $e \in E(\Gamma)$, we write $i(e)$ for the initial vertex of $e$, and $\overline{e}$ for the edge $e$ with the orientation reversed, so that $e$ has endpoints $i(e)$ and $i(\overline{e})$. 

Let $I$ be a compact interval. A \emph{trivial path} in a graph $\Gamma$ is a map $\rho \colon I \to \Gamma$ whose image is a single point in $\Gamma$. A \emph{path} in $\Gamma$ is a trivial path or an immersion $\rho \colon I \to \Gamma$. We will not distinguish between a path $\rho$ and an orientation-preserving reparametrization of $\rho$. Any continuous map $\tau \colon I \to \Gamma$ is homotopic relative to its endpoints to a unique path which we denote by $[\tau]$ and call the \emph{tightening} of $\tau$. We will sometimes use the term \emph{path} to mean the image of a path.

A \emph{tree} is a connected graph with no cycles. We say that an action of a group $G$ on a tree $\mathcal{T}$ is \emph{trivial} if there exists a global fixed point, and \emph{minimal} if there does not exist a proper $G$-invariant subtree. An action of $G$ on $\mathcal{T}$ is said to be \emph{without inversion} if whenever $G$ fixes an edge, it acts trivially on each endpoint of the edge. We define an action of $G$ on a tree $\mathcal{T}$ to be \emph{irreducible} if $G$ does not fix a point, a line, or an end of the tree. 

A \emph{$G$-tree} is a simplicial tree $\mathcal{T}$ together with an action of the group $G$ which is minimal and without inversion. We will also sometimes use the term \emph{splitting of $G$} to mean a $G$-tree. A \emph{free splitting} is a splitting with trivial edge stabilisers.

Following Sela \cite{Sela1997}, given a non-negative integer $k$, we say an action of a torsion-free group $G$ on a tree $\mathcal{T}$ is \emph{$k$-acylindrical} if the pointwise stabiliser of any segment of length greater than $k$ is trivial. We say the action is \emph{acylindrical} if it is $k$-acylindrical for some $k$. 

\begin{remark}
    Given a Gromov hyperbolic space $X$, there is an alternative notion of an acylindrical action $G \curvearrowright X$ due to Bowditch \cite{Bowditch2008,Osin2016}, and a non-(virtually cyclic) group admitting such an action with unbounded orbits is said to be \emph{acylindrically hyperbolic}. Even though, in general, the two notions of an acylindrical action $G \curvearrowright X$ are different, they do coincide in the case when $X$ is a tree and $G$ is torsion-free \cite[Lemma~4.2]{MinasyanOsin2015}. In this paper we will only be considering acylindrical actions of free-by-cyclic groups on trees, and so we do not introduce the more general terminology of an acylindrical action in the sense of Bowditch.
\end{remark}

A \emph{graph of groups} is a finite connected graph $\mathcal{G}$ together with collections of groups $\{ G_v \mid v \in V(\mathcal{G}) \}$ and $\{ G_e \mid e \in E(\mathcal{G}) \}$, with $G_e = G_{\overline{e}}$, and a collection of injective group homomorphisms $\{ \iota_e\colon G_e \to G_{i(e)} \mid e \in E(\mathcal{G}) \}$. The \emph{fundamental group} of a graph of groups $\mathcal{G}$ is the group
\[
\pi_1(\mathcal{G}) \coloneqq \left( *_{v \in V(\mathcal{G})} G_v \right) * F_E \Big/
\left\langle\!\!\left\langle \{ \iota_e(h)^{-1} e \iota_{\overline{e}}(h) e^{-1} \mid e \in E(\mathcal{G}), h \in G_e \} \cup \{ e\overline{e} \mid e \in E(\mathcal{G}) \} \cup \mathcal{T}_0 \right\rangle\!\!\right\rangle
\]
where $F_E$ is a free group with basis $E(\mathcal{G})$, and $\mathcal{T}_0 \subseteq E(\mathcal{G}) \subseteq F_E$ is a spanning tree in the graph~$\mathcal{G}$. By Bass--Serre theory, given a group $G$ there is a one-to-one correspondence between graphs of groups $\mathcal{G}$ such that $G \cong \pi_1(\mathcal{G})$ and cocompact $G$-trees $\mathcal{T}_{\mathcal{G}}$, so that the vertex and edge stabilisers in $\mathcal{T}_{\mathcal{G}}$ are conjugates of the vertex groups $G_v$ and the edge groups $G_e$ of $\mathcal{G}$, respectively. The fact that $\mathcal{T}_{\mathcal{G}}$ is in fact a tree is a geometric analogue of the following theorem.

\begin{thm}[Normal form theorem; see {\cite[Chapter I, Theorem 4.1]{DicksDunwoody1989}}] \label{thm:gog-normal-forms}
Let $\mathcal{G}$ be a graph of groups, and let $g = g_0 e_1 g_1 \cdots e_n g_n \in \pi_1(\mathcal{G})$ with $g_j \in G_{v_j}$ for some $v_j \in V(\mathcal{G})$ and $e_j \in E(\mathcal{G})$, such that $v_n=v_0$, $v_{j-1} = i(e_j)$ and $v_j = i(\overline{e_j})$ for $1 \leq j \leq n$. If $g = 1$ in $\pi_1(\mathcal{G})$, then either $n = 0$ and $g_0 = 1$ in $G_{v_0}$, or $n \geq 2$ and there exists $j \in \{1,\ldots,n-1\}$ such that $e_{j+1} = \overline{e_j}$ and $g_j \in \iota_{\overline{e_j}}(G_{e_j})$.
\end{thm}

\subsection{Equationally Noetherian groups} 

Let $\mathbb{F}_m$ be a free group with basis $X = \{ x_1,\ldots,x_m \}$. Given an element $s \in \mathbb{F}_m$, a group $G$ and a tuple $(g_1,\ldots,g_m) \in G^m$, we write $s(g_1,\ldots,g_m)$ for the element of $G$ obtained by substituting $g_i$ for each copy of $x_i$ in $s$---that is, for the image of $s$ under the group homomorphism $\mathbb{F}_m \to G$ sending $x_i \mapsto g_i$. We call an element $s \in \mathbb{F}_m$ an \emph{equation}, and a subset $S \subseteq \mathbb{F}_m$ a \emph{system of equations}. For any $S \subseteq \mathbb{F}_m$ and any group $G$, the \emph{solution set} of $S$ in $G$ is defined as
\[
V_G(S) \coloneqq \{ (g_1,\ldots,g_m) \in G^m \mid s(g_1,\ldots,g_m) = 1 \text{ for all } s \in S \}.
\]
We say a group $G$ is \emph{equationally Noetherian} if for every positive integer $m$ and every $S \subseteq \mathbb{F}_m$ there exists a finite subset $S_0 \subseteq S$ such that $V_G(S_0) = V_G(S)$.

More generally, given a subset $S \subseteq \mathbb{F}_m$, a group $G$ and a subset $A \subseteq G$, we may define
\[
V_{G,A}(S) \coloneqq \{ (g_1,\ldots,g_m) \in G^m \mid s(g_1,\ldots,g_m) \in A \text{ for all } s \in S \}.
\]
We say the subset $A$ is \emph{quasi-algebraic} in $G$ if for every positive integer $m$ and every $S \subseteq \mathbb{F}_m$ there exists a finite subset $S_0 \subseteq S$ such that $V_{G,A}(S_0) = V_{G,A}(S)$. In particular, $G$ is equationally Noetherian if and only if $\{1\}$ is quasi-algebraic in $G$.

An alternative view of equationally Noetherian groups, taken in \cite{GrovesHull2019}, is given by considering sequences of homomorphisms $\mathbb{F}_m \to G$, as follows. Let $\omega$ be a non-principal ultrafilter (on $\NN$), i.e.\ a finitely additive probability mean $\mathcal{P}(\NN) \to \{0,1\}$. Given a property $P = P_j$ that depends on $j \in \NN$, we say that $P_j$ holds \emph{$\omega$-almost surely} if $\omega(\{ j \in \NN \mid P_j \text{ holds} \}) = 1$. Given a sequence of homomorphisms $(\varphi_j\colon \mathbb{F}_m \to G)_{j \in \NN}$ from a free group $\mathbb{F}_m$ to a group $G$, we define $\ker^\omega(\varphi_j) \coloneqq \{ g \in \mathbb{F}_m \mid \varphi_j(g) = 1 \text{ $\omega$-almost surely} \}$, which is clearly a normal subgroup of $\mathbb{F}_m$; more generally, given a subset $A \subseteq G$, we define $\varphi_{j,\omega}^{-1}(A) \coloneqq \{ g \in \mathbb{F}_m \mid \varphi_j(g) \in A \text{ $\omega$-almost surely} \}$. We then have the following alternative description of equationally Noetherian groups and quasi-algebraic subsets; the statement about equational Noetherianity is \cite[Theorem~3.5]{GrovesHull2019}, and the statement about quasi-algebraicity is a straightforward generalisation of this.

\begin{thm} \label{thm:en-characterisation}
Let $G$ be a group, $A \subseteq G$ a subset, and $\omega$ a non-principal ultrafilter. Then $A$ is quasi-algebraic in $G$ if and only if for every rank $m$ and any sequence of homomorphisms $(\varphi_j\colon \mathbb{F}_m \to G)_{j \in \NN}$ we have $\varphi_{j,\omega}^{-1}(A) \subseteq \varphi_j^{-1}(A)$ $\omega$-almost surely. In particular, $G$ is equationally Noetherian if and only if for any $(\varphi_j\colon \mathbb{F}_m \to G)_{j \in \NN}$ we have $\ker^\omega(\varphi_j) \subseteq \ker(\varphi_j)$ $\omega$-almost surely.
\end{thm}

The key benefit of this construction is that it allows more geometric arguments in the presence of acylindrical hyperbolicity. In particular, suppose that a group $G$ acts acylindrically on a hyperbolic space $\mathcal{T}$; in this paper, we will only be using the case when $G$ is torsion-free and $\mathcal{T}$ is a tree. A sequence of homomorphisms $(\varphi_j\colon \mathbb{F}_m \to G)_{j \in \NN}$ is then said to be \emph{non-divergent} (with respect to $\omega$ and the action on $\mathcal{T}$) if
\[
\lim_{j \to \omega} \inf_{v \in \mathcal{T}} \max_{1 \leq i \leq m} d_{\mathcal{T}}(v,\varphi_j(x_i) \cdot v) < \infty,
\]
where $\{ x_1,\ldots,x_m \}$ is a free basis for $\mathbb{F}_m$. The main technical result of \cite{GrovesHull2019}, and the main tool we will be using in the proof of Theorem~\ref{thm:main}, then states that in order to show $G$ is equationally Noetherian it is enough to consider non-divergent sequences of homomorphisms in Theorem~\ref{thm:en-characterisation}, as follows.

\begin{thm}[{\cite[Theorem~B]{GrovesHull2019}}] \label{thm:en-non-divergent}
Let $G$ be a group acting acylindrically on a hyperbolic space $\mathcal{T}$ with unbounded orbits, and let $\omega$ be a non-principal ultrafilter. Suppose that for any rank $m$ and any non-divergent sequence of homomorphisms $(\varphi_j\colon \mathbb{F}_m \to G)_{j \in \NN}$ we have $\ker^\omega(\varphi_j) \subseteq \ker(\varphi_j)$ $\omega$-almost surely. Then $G$ is equationally Noetherian.
\end{thm}

\subsection{Dynamics of free group automorphisms and topological representatives}
\label{s:autos}

Let $F$ denote a free group with a free basis $X$. For any element $g\in F$, we denote by $|g|_X$ the length of the reduced word representative of $g$. We write $|\bar{g}|_X$ to denote the minimal length of a cyclically reduced word representing a conjugate of $g$. 

An outer automorphism $\Phi \in \Out(F)$ acts on the set of conjugacy classes of elements in $F$. Given a conjugacy class $\bar{g}$ of an element $g\in F$, we say that \emph{$\bar{g}$ grows polynomially of degree $d$ under the iteration of $\Phi$,} if there exist constants $A,B > 0$ such that for all $k \geq 1$ \[Ak^d - A\leq |\Phi^k(\bar{g})|_X \leq Bk^d .\]
We say \emph{$g$ grows exponentially} if there exists a constant $\mu > 1$ such that for all $k \geq 1$ \[\mu^k \leq|\Phi^k(\bar{g})|.\]
For any two free generating sets $X$ and $Y$ of $F$, the word metrics with respect to $X$ and $Y$ are bi-Lipschitz equivalent. It follows that the growth of a conjugacy class under $\Phi$ does not depend on the specific choice of free basis for $F$.

We say the outer automorphism $\Phi \in \Out(F)$ \emph{grows polynomially of degree $d$} if there exists an element $g\in F$ whose conjugacy class grows polynomially of degree $d$ and such that every other conjugacy class grows polynomially of degree $\leq d$. We say $\Phi$ grows \emph{exponentially} if there exists an exponentially-growing conjugacy class.

We remark that by the existence of relative train track representatives \cite{BestvinaHandel1992}, every outer automorphism $\Phi \in \Out(F)$ grows either exponentially or polynomially.

We recall the following elementary fact:
\begin{lem}\label{lemma:slow-growth}
    An element $\Phi \in \Out(F)$ has polynomial growth of degree $0$ if and only if $\Phi$ has finite order in $\Out(F)$.
\end{lem}

An outer automorphism $\Phi \in \Out(F)$ is said to be \emph{unipotent} if it induces a unipotent element of $\GL(m, \mathbb{Z})$. We will often abbreviate unipotent polynomially growing to \emph{UPG}.

\begin{lem}[{\cite[Corollary~5.7.6]{BestvinaFeighnHandel2000}}] \label{lem:upg-finite-index}
    Let $\Phi \in \Out(F)$ be a polynomially growing outer automorphism. Then there exists a positive integer $k$ such that $\Phi^k$ is UPG with growth of the same degree as that of $\Phi$.
\end{lem}

\medskip

A \emph{topological representative} of an outer automorphism $\Phi \in \Out(F)$ is a tuple $(f, \Gamma)$, where $\Gamma$ is a connected graph and $f \colon \Gamma \to \Gamma$ is a homotopy equivalence which determines $\Phi$. Moreover, we assume that $f$ preserves the set of vertices of $\Gamma$ and that it is locally injective on the interiors of edges. 

We call a topological representative $(f, \Gamma)$ \emph{upper triangular} if $f$ fixes every vertex of $\Gamma$ and if $(f, \Gamma)$ admits a maximal filtration \[\emptyset = \Gamma^{(0)} \subseteq \Gamma^{(1)} \subseteq \ldots \subseteq \Gamma^{(l)} = \Gamma,\] by subgraphs, such that the following properties hold for every index $i$:
\begin{enumerate}[label=(\alph*)]
\item $f(\Gamma^{(i)}) \subseteq \Gamma^{(i)}$;
\item the graph $\Gamma^{(i)}$ is obtained from $\Gamma^{(i-1)}$ by adding a single edge $e_i$;
\item the map $f$ sends the edge $e_i$ to the concatenation $e_i \cdot \rho_i$, where $\rho_i$ is a closed path in the subgraph $\Gamma^{(i-1)}$.
\end{enumerate}

Bestvina--Feighn--Handel introduce a family of topological representatives for polynomially growing outer automorphisms known as \emph{unipotent representatives} in \cite[Definition~3.13]{BestvinaFeighnHandel2005}. These are relative train tracks (in the sense of \cite{BestvinaHandel1992}) which are upper triangular, and which satisfy additional conditions allowing for tighter control of certain path splittings. We will not need the definition in full. Instead, we cite the following proposition which will be useful to us in what follows: 

\begin{prop}[{\cite[Proposition~3.18]{BestvinaFeighnHandel2005}}]\label{prop:eigen-rays}
    Let $\Phi \in \Out(F)$ be a UPG element, and let $(f, \Gamma)$ be a unipotent representative. Let $e_i$ be an edge of $\Gamma$ and $\rho_i$ the path in $\Gamma$ such that $f(e_i) = e_i\cdot \rho_i$. Suppose that $[f(\rho_i)] \neq \rho_i$. Then, for any initial segment $\tau$ of the infinite ray $e_i  \rho_i [f(\rho_i)] [f^2(\rho_i)] \ldots$ and any path $\sigma$ in $\Gamma$ which traverses $e_i$ in either direction, we have that $\tau$ or $\bar{\tau}$ is a subpath of $[f^r(\sigma)]$, for all $r$ sufficiently large.
\end{prop}

\begin{thm}[{\cite[Theorem~5.1.8]{BestvinaFeighnHandel2000}}]
    Every UPG element in $\Out(F)$ admits a unipotent representative.
\end{thm}

\section{Acylindrical actions}
\label{s:acylindrical}

The aim of this section is to give a common framework for showing that certain actions of free-by-cyclic groups on trees are acylindrical. We will use this to prove that the actions obtained in the work of Guirardel--Horbez in \cite[Section~6]{GuirardelHorbez2021} and Andrew--Martino \cite[Proposition~5.2.2]{AndrewMartino2022} (see Lemma~\ref{lem:linear-growth-splitting}) are 4-acylindrical.

Let $\Phi \in \Out(F)$ be an outer automorphism. In order to be consistent with the terminology established in \cite{GuirardelHorbez2021}, we will consider the mapping torus of an element in $\Out(F)$ as a subgroup of $\Aut(F)$. To that end, let $G_{\Phi}$ denote the preimage of the cyclic subgroup generated by $\Phi$ in $\Out(F)$ under the natural quotient map $\Aut(F) \twoheadrightarrow \Out(F)$. Then for any automorphism $\varphi$ representing $\Phi$ we have that \[G_{\Phi} = \langle \Inn(F), \varphi \rangle \cong F \rtimes_{\Phi} \ZZ.\] 

A \emph{$\Phi$-invariant splitting of $F$} is an $F$-tree $\mathcal{T}$, such that for each element $\alpha \in G_{\Phi}$ there is an isometry $I_\alpha \colon \mathcal{T} \to \mathcal{T}$ so that for any $g \in F$ and $x \in \mathcal{T}$, 
 \[I_{\alpha}(gx) = \alpha(g) I_{\alpha}(x).\] We will moreover assume that the $F$-action on $\mathcal{T}$ is irreducible. Then, the assignment $\alpha \mapsto I_{\alpha}$ determines an action of $G_{\Phi}$ on the tree $\mathcal{T}$.

\begin{defn}[Aligned splittings]\label{defn:aligned_splitting}
    Let $\Phi \in \Out(F)$. A \emph{$\Phi$-aligned tree} $\mathcal{T}$ is a $\Phi$-invariant bipartite $F$-tree $(\mathcal{T}, V_0(\mathcal{T}), V_1(\mathcal{T}))$, such that the following properties are satisfied.
     \begin{enumerate}[label=(\alph*)]
        \item\label{it:aligned_splitting_ffs} For any two distinct vertices $w_1, w_2 \in V_0(\mathcal{T})$, the intersection $\Stab_{G_{\Phi}}(w_1) \cap \Stab_{G_{\Phi}}(w_2) \cap \Inn(F)$ is trivial.
     
        \item\label{it:aligned_splittings_black} There exists an injective map  
        \[\begin{split}
            V_1(\mathcal{T}) &\to \pi^{-1}(\Phi) \\
             v&\mapsto \varphi_v
        \end{split}\]
        where $\pi \colon \Aut(F) \twoheadrightarrow \Out(F)$ is the natural quotient. The automorphism $\varphi_v$, considered as an element of $G_{\Phi}$, fixes the ball of radius one around $v$ pointwise.
       
     \end{enumerate}
     We say that a tree $\mathcal{T}$ is \emph{aperiodic $\Phi$-aligned} if it is $\Phi$-aligned and if for any two vertices $v_0$ and $v_1$ in $V_{1}(\mathcal{T})$ such that $\varphi_{v_0}^k = \varphi_{v_1}^k$ for some positive integer $k$, we have that $\varphi_{v_0} = \varphi_{v_1}$.
\end{defn}

\begin{prop}\label{prop:aligned_acylindrical_trees}
    The action of $G_{\Phi}$ on a non-trivial aperiodic $\Phi$-aligned tree $\mathcal{T}$ is $4$-acylindrical.
\end{prop}

 \begin{proof}
     Let $L$ be a line segment of length $5$ in $\mathcal{T}$ and suppose that $L$ is fixed pointwise by a non-trivial element $g \in G_{\Phi}$. Let $v\in V_1(\mathcal{T})$ and suppose that $v$ lies between two vertices $w_0,w_1 \in V_0(\mathcal{T})$ in $L$. Let $\varphi_v \in G_{\Phi}$ denote a lift of $\Phi$ which fixes the ball of radius one around~$v$. 
     
     The element $g$ can be expressed as $g = \ad_a \cdot \varphi_v^{k}$, for some $a \in F$ and $k \in \ZZ$. Let $L_0$ be the subsegment of $L$ of length two which consists of the vertices $w_{0}$, $v$ and $w_{1}$, and the edges between them. Since $\varphi_v$ fixes $L_0$ pointwise, it must be the case that the element $g \cdot \varphi_v^{-k} = \ad_a$ also fixes $L_0$ pointwise. However, the intersection $\Stab(w_{0}) \cap \Stab(w_{1}) \cap \Inn(F)$ is trivial. It follows that the element $a$ is trivial and $g = \varphi_v^k$, for some non-zero integer $k$.

     Since $L$ is a segment of length $5$, there exist two distinct vertices $v_1, v_2 \in V_1(\mathcal{T})$ in $L$, each of which lies between two vertices in $V_0(\mathcal{T})$ in $L$. Hence for each $i = 1,2$, there exists an integer $k_i \in \ZZ$ such that $g = \varphi_i^{k_i}$, where $\varphi_i$ a lift of $\Phi$ which fixes a ball of radius one around the vertex $v_i$. Hence $\varphi_1^{k_1} = \varphi_2^{k_2}$. Since $\varphi_1$ and $\varphi_2$ are lifts of $\Phi$, we must have that $k_1 = k_2$. By aperiodicity, we must have that $\varphi_1 = \varphi_2$. However, this contradicts injectivity in Item~\ref{it:aligned_splittings_black} in the definition of $\Phi$-aligned splittings.
\end{proof}

An outer automorphism $\Phi \in \Out(F)$ is said to be \emph{neat} if for every element $x \in F$ and every automorphism $\varphi$ which represents $\Phi$, if $\varphi^k(x) = x$ for some non-zero integer $k$ then $\varphi(x) = x$.

\begin{lem}\label{lemma:non-periodic_trees}
    If $\Phi \in \Out(F)$ is neat then every $\Phi$-aligned tree $\mathcal{T}$ is aperiodic.
\end{lem}

\begin{proof}
    Let $\varphi_1$ and $\varphi_2$ be two lifts of $\Phi$ in $G_{\Phi}$ such that $\varphi_1^k = \varphi_2^k$ for some non-zero integer $k$. Since each $\varphi_i$ is a lift of $\Phi$, we must have that $\varphi_1 = \ad_u \cdot \varphi_2$ for some $u \in F$. It follows that $\varphi_2^k = (\ad_u \cdot\varphi_2)^k$, and hence 
     \[ u  \varphi_2(u) \cdots \varphi_2^{k-1}(u) = 1.\]
     Thus, $\varphi_2^k(u) = u$. Since $\Phi$ is neat, it follows that $\varphi_2(u) = u$ and $u^k = 1$, and hence $u$ is the trivial element in $F$. In particular, we must have that $\varphi_1 = \varphi_2$.
\end{proof}

\begin{lem}[Bestvina--Feighn--Handel {\cite[Proposition~4.41]{BestvinaFeighnHandel2005}}]
    If $\Phi \in \Out(F)$ is UPG then it is neat.
\end{lem}

Let $\IA$ denote the finite index normal subgroup of $\Out(F)$ which is the kernel of the natural quotient map \[\Out(F) \twoheadrightarrow \GL(n, \ZZ / 3\ZZ).\] 

 \begin{thm}[{Guirardel--Horbez \cite[Theorem~1.8]{GuirardelHorbez2021}}]\label{thm:GH}
For every outer automorphism $\Phi \in \IA$, there exists a $\Phi$-aligned $F$-tree $\mathcal{T}_{\Phi}$. The tree $\mathcal{T}_{\Phi}$ is non-trivial if $\Phi$ has infinite order in $\Out(F)$ and $F$ admits a non-trivial $\Phi$-invariant free splitting. Moreover, the assignment 
\[\Phi \mapsto \mathcal{T}_{\Phi} \]is $\Out(F)$-equivariant, where the action of $\Out(F)$ on $\IA$ is by conjugation. 
\end{thm}

\begin{proof}
    Let $H$ be the cyclic subgroup of $\Out(F)$ generated by $\Phi$. Let $(U_H^1, V_0(\mathcal{T}), V_1(\mathcal{T}))$ be the bipartite $\Phi$-invariant tree obtained in \cite[Theorem~6.12]{GuirardelHorbez2021}. Let us write $\mathcal{T}_{\Phi}$ to denote $U_H^1$.
    
    By construction of the tree in \cite[Theorem~6.12]{GuirardelHorbez2021}, the groups $\Stab_{G_{\Phi}}(w) \cap \Inn(F)$, for $w \in V_0(\mathcal{T})$, are exactly the vertex stabilisers of a free splitting of the free group $\Inn(F) \cong F$. Hence, for any two distinct vertices $w_1, w_2 \in V_0(\mathcal{T})$, the intersection  $\Stab_{G_{\Phi}}(w_1) \cap \Stab_{G_{\Phi}}(w_2) \cap \Inn(F)$ is trivial.

    Let $\mathcal{T}$ be a maximal $\Phi$-invariant free $F$-splitting. Such a splitting exists by \cite[Proposition~6.2]{GuirardelHorbez2021}. The action of $F \simeq \Inn(F)$ extends to an action of $G_{\Phi}$ on $\mathcal{T}$ with infinite cyclic edge stabilisers. The induced $G_{\Phi}$-action on the quotient $\mathcal{T} / F$ is trivial by \cite[Lemma~2.12]{GuirardelHorbez2021}, and hence the $G_{\Phi}$-stabiliser of each edge of $\mathcal{T}$ is generated by a lift of $\Phi$ in $G_{\Phi}$. Let $\sim$ denote the equivalence relation on the set of edges of $\mathcal{T}$ such that two edges are defined to be $\sim$-equivalent whenever their $G_{\Phi}$-stabilisers are equal. 

    By construction, the vertices in $V_1(\mathcal{T})$ correspond to $\sim$-equivalence classes of edges in $\mathcal{T}$. Let us define a map 
     \[\begin{split}
            V_1(\mathcal{T}) &\to \pi^{-1}(\Phi) \\
             v&\mapsto \varphi_v
        \end{split}\]
    where $\varphi_v$ denotes the representative of $\Phi$ which generates the stabiliser of any edge in the equivalence class of the vertex $v$. This map is clearly well-defined and injective. Suppose that $u \in V_0(\mathcal{T})$ is adjacent to the vertex $v \in V_1(\mathcal{T})$ in $\mathcal{T}_{\Phi}$. Then $u$ is a vertex in the original tree $\mathcal{T}$ which is adjacent to an edge $e$ in $\mathcal{T}$ whose stabiliser is given by $\Stab_{G_{\Phi}}(e) = \langle \varphi_v \rangle$. But since $u$ is adjacent to the edge $e$, $\varphi_v$ also stabilises the vertex $u$.

    This proves that the tree $(\mathcal{T}_{\Phi}, V_0(\mathcal{T}), V_1(\mathcal{T}))$ is $\Phi$-aligned. The remaining claims follow from \cite[Theorem~1.8]{GuirardelHorbez2021}.
\end{proof}
 
\begin{prop}\label{prop:acylindrical_actions}
    Let $\Phi \in \Out(F)$ be an infinite order polynomially growing outer automorphism. Then $G_{\Phi}$ admits a $4$-acylindrical action on a tree.
\end{prop}

\begin{proof}
    Let $k$ be a sufficiently high positive integer such that $\Phi^k$ is contained in $\IA$. Then, $\Phi^k$ is UPG by \cite[Corollary~5.7.6]{BestvinaFeighnHandel2000} and thus $F$ admits a non-trivial $\Phi^k$-invariant free splitting. Hence by Theorem~\ref{thm:GH}, there exists a non-trivial $\Phi^k$-aligned tree $\mathcal{T}$. By Lemma~\ref{lemma:non-periodic_trees}, since the outer automorphism $\Phi^k$ is UPG and thus neat, the tree $\mathcal{T}$ is aperiodic. By Proposition~\ref{prop:aligned_acylindrical_trees}, the action of $G_{\Phi^k}$ on $\mathcal{T}$ is 4-acylindrical. By equivariance in Theorem~\ref{thm:GH}, the tree $\mathcal{T}$ is $\Phi$-invariant. Hence, the 4-acylindrical action of $G_{\Phi^k}$ extends to a 4-acylindrical action of $G_{\Phi}$ on~$\mathcal{T}$.
\end{proof}

\section{Linear growth}
\label{s:linear}

The goal in this section is to prove the following result, which will be taken as the base case of induction in the proof of Theorem~\ref{thm:main}.

\begin{prop} \label{prop:en-linear}
Let $F$ be a finitely generated free group, and let $\Phi \in \Out(F)$ be a linearly growing UPG outer automorphism. Then $G = F \rtimes_\Phi \ZZ$ is equationally Noetherian.
\end{prop}

We will prove Proposition~\ref{prop:en-linear} by showing that the groups $G = F \rtimes_\Phi \ZZ$ as above satisfy the assumptions of the following proposition.

\begin{prop} \label{prop:cyclic-ext}
Let $G$ be a group and $F \unlhd G$ a normal subgroup such that both $F$ and $G/F$ are equationally Noetherian. Suppose that $G$ splits as the fundamental group of a graph of groups $\mathcal{G}$ in which the group associated to any vertex $v \in V(\mathcal{G})$ is $F_v \times Z_v$, where
\begin{enumerate}
\item \label{it:cyclic-qalg} $F_v$ is a quasi-algebraic subgroup of $F$;
\item \label{it:cyclic-Zv} $Z_v \cap F = \{1\}$ and $Z_v F = G$ (in particular, $Z_v \cong G/F$);
\item \label{it:cyclic-He} if $G_e$ is the group associated to an edge $e \in E(\mathcal{G})$, then $Z_{i(e)} \subseteq \iota_e(G_e)$;
\item \label{it:cyclic-edges-in-G} $e \in F$ (as an element of $G = \pi_1(\mathcal{G})$) for all $e \in E(\mathcal{G})$;
\item \label{it:cyclic-acyl} the action of $G$ on the Bass--Serre tree corresponding to $\mathcal{G}$ is acylindrical.
\end{enumerate}
Then $G$ is equationally Noetherian.
\end{prop}

Our main tool for proving Proposition~\ref{prop:cyclic-ext} is the following technical lemma, which in conjunction with Theorem~\ref{thm:en-non-divergent} will imply the result of the proposition.

\begin{lem} \label{lem:sorry}
Let $G$, $F$, $\mathcal{G}$, $F_v$, $Z_v$ and $G_e$ be as in Proposition~\ref{prop:cyclic-ext}. Let $\FF{v}$ and $\FFF{v}$ (for $v \in V(\mathcal{G})$) be finitely generated free groups, and let $\mathbb{F} = \left( *_{v \in V(\mathcal{G})} (\FF{v} * \FFF{v}) \right) * \FF{E}$, where $\FF{E}$ is a free group with free basis $\{ X_e \mid e \in E(\mathcal{G}) \}$. Suppose that $(\varphi_j\colon \mathbb{F} \to G)_{j \in \mathbb{N}}$ is a sequence of homomorphisms such that $\varphi_j(\FF{v}) \subseteq F_v$, $\varphi_j(\FFF{v}) \subseteq Z_v$ and $\varphi_j(X_e) = e$, and let $\omega$ be a non-principal ultrafilter. Then $\ker^\omega(\varphi_j) \subseteq \ker(\varphi_j)$ holds $\omega$-almost surely.
\end{lem}

We prove Lemma~\ref{lem:sorry} by extending the $\varphi_j$ to homomorphisms $\overline\varphi_j \colon \overline{\mathbb{F}} \to G$, where $\overline{\mathbb{F}}$ is a free group containing $\mathbb{F}$ as a free factor, and explicitly describing a subset $\mathcal{K} \subseteq \overline{\mathbb{F}}$ whose normal closure contains $\ker^\omega(\varphi_j)$ and is $\omega$-almost surely contained in $\ker(\overline\varphi_j)$. We then take an element $g \in \overline{\mathbb{F}}$ such that $\overline\varphi_j(g) = 1$ $\omega$-almost surely, and consider a ``canonical'' expression $g_{0,j}e_1g_{1,j} \cdots e_ng_{n,j}$ of $\overline\varphi_j(g)$ as in Theorem~\ref{thm:gog-normal-forms}, where the elements $g_{i,j}$ and $e_i$ are $\overline\varphi_j$-images of subwords of $g$ (i.e.\ the reduced word representing $g$), so that $e_1,\ldots,e_n \in E(\mathcal{G})$ are independent of $j$. The key point is that the assumptions \ref{it:cyclic-qalg}--\ref{it:cyclic-edges-in-G} in Proposition~\ref{prop:cyclic-ext}, together with the fact that $F$ and $G/F$ are equationally Noetherian, allow choosing $\mathcal{K}$ in such a way that inserting an element from $\mathcal{K}$ in the reduced word representing $g$ allows us to $\omega$-almost surely reduce the length $n$ of this canonical expression of $\overline\varphi_j(g)$.  The assumption~\ref{it:cyclic-acyl} in Proposition~\ref{prop:cyclic-ext} is not used in the proof of Lemma~\ref{lem:sorry} and only appears later, in the proof of Proposition~\ref{prop:en-linear}.

\begin{proof}[Proof of Lemma~\ref{lem:sorry}]
Let $S_v$ be a free basis for $\FFF{v}$ for all $v \in V(\mathcal{G})$, let $\widetilde{S} = \bigsqcup_{v \in V(\mathcal{G})} S_v$, and let $\widetilde{\mathbb{F}}$, $\FFFF{v}$ and $\FFFF{e}$ be free groups with free bases $\widetilde{S}$, $\{ s_v \mid s \in \widetilde{S} \}$ and $\{ s_e \mid s \in \widetilde{S} \}$, respectively, where $v \in V(\mathcal{G})$ and $e \in E(\mathcal{G})$. Note that $\FFF{v}$ can be canonically identified with the free factor of $\FFFF{v}$ generated by $\{ s_v \mid s \in S_v \}$. Let $\widehat{\mathbb{F}} = \left( *_{v \in V(\mathcal{G})} \FF{v} \right) * \left( *_{e \in E(\mathcal{G})} \FFFF{e} \right) * \FF{E}$, and let $\overline{\mathbb{F}} = \widehat{\mathbb{F}} * \left( *_{v \in V(\mathcal{G})} \FFFF{v} \right)$.

We extend each $\varphi_j\colon \mathbb{F} \to G$ to a homomorphism $\overline\varphi_j \colon \overline{\mathbb{F}} \to G$ as follows. We define $\overline\varphi_j \colon \FFFF{v} \to G$ (for $v \in V(\mathcal{G})$) by setting, for each $s \in S_w \subseteq \widetilde{S}$, the element $\overline\varphi_j(s_v)$ to be an element in $Z_v \cap \varphi_j(s)F$, which exists and is unique by \ref{it:cyclic-Zv}. We then define $\overline\varphi_j \colon \FFFF{e} \to G$ (for $e \in E(\mathcal{G})$) by setting, for each $s \in \widetilde{S}$, $\overline\varphi_j(s_e) \coloneqq \overline\varphi_j(s_{i(e)})^{-1} e \overline\varphi_j(s_{i(\overline{e})}) e^{-1}$; note that since we have $e \in F$ by \ref{it:cyclic-edges-in-G} and, by construction, $\overline\varphi_j(s_{i(e)}) F = \overline\varphi_j(s_{i(\overline{e})}) F$, it follows that $\overline\varphi_j(s_e) \in F$. Finally, we define the restriction of $\overline\varphi_j$ to $\left( *_{v \in V(\mathcal{G})} \FF{v} \right) * \FF{E}$ to be equal to the corresponding restriction of $\varphi_j$.

Since $\overline\varphi_j$ extends $\varphi_j$, we have $\ker(\varphi_j) = \mathbb{F} \cap \ker(\overline\varphi_j)$, so it is enough to show that $\ker^\omega(\varphi_j) \subseteq \ker(\overline\varphi_j)$ $\omega$-almost surely. We define
\begin{align*}
\mathcal{K}_0 &= \{ X_e \mid e \in E(\mathcal{G}), e=1 \text{ in } G \} \cup \{ X_eX_{\overline{e}} \mid e \in E(\mathcal{G}) \}, \\
\mathcal{K}_1 &= \ker^\omega(\overline\varphi_j|_{\widehat{\mathbb{F}}}) \cup \bigcup_{v \in V(\mathcal{G})} \ker^\omega(\overline\varphi_j|_{\FFFF{v}}), \\
\mathcal{K}_2 &= \{ s_{i(e)} s_e X_e s_{i(\overline{e})}^{-1} X_e^{-1} \mid e \in E(\mathcal{G}) \}, \qquad \text{and} \\
\mathcal{K}_3 &= \bigcup_{v \in V(\mathcal{G})} \left\{ [g,h] \:\middle|\: g \in \FFFF{v}, h \in \widehat{\mathbb{F}}, \overline\varphi_j(h) \in F_v \text{ $\omega$-almost surely} \right\},
\end{align*}
and let $\mathcal{K} = \mathcal{K}_0 \cup \mathcal{K}_1 \cup \mathcal{K}_2 \cup \mathcal{K}_3$. We then have $\mathcal{K} \subseteq \ker(\overline\varphi_j)$ $\omega$-almost surely: indeed, we have $\mathcal{K}_0 \cup \mathcal{K}_2 \subseteq \ker(\overline\varphi_j)$ for all $j$ by construction; we have $\mathcal{K}_1 \subseteq \ker(\overline\varphi_j)$ $\omega$-almost surely since the $\overline\varphi_j$-image of $\widehat{\mathbb{F}}$ (respectively $\FFFF{v}$) is contained in the equationally Noetherian group $F$ (respectively $Z_v \cong G/F$); and we have $\mathcal{K}_3 \subseteq \ker(\overline\varphi_j)$ $\omega$-almost surely since, given $v \in V(\mathcal{G})$, $F_v$ is quasi-algebraic in $F$ by \ref{it:cyclic-qalg}, and consequently
\[
\{ h \in \widehat{\mathbb{F}} \mid [\FFFF{v},h] \subseteq \mathcal{K}_3 \} = (\overline\varphi_j|_{\widehat{\mathbb{F}}})_\omega^{-1}(F_v) \subseteq (\overline\varphi_j|_{\widehat{\mathbb{F}}})^{-1}(F_v) \subseteq \{ h \in \widehat{\mathbb{F}} \mid [\FFFF{v},h] \subseteq \ker(\overline\varphi_j) \}
\]
holds $\omega$-almost surely. In particular, we have $\langle\!\langle \mathcal{K} \rangle\!\rangle \subseteq \ker(\overline\varphi_j)$ $\omega$-almost surely, and so it is enough to show that we have $\ker^\omega(\varphi_j) \subseteq \langle\!\langle \mathcal{K} \rangle\!\rangle$.

Thus, let $g \in \mathbb{F}$ be such that $\varphi_j(g) = 1$ $\omega$-almost surely; we claim that $g \in \langle\!\langle \mathcal{K} \rangle\!\rangle$. After inserting letters of $\mathcal{K}_0$ in a word representing $g$ (that is, multiplying $g$ by an element of $\langle\!\langle \mathcal{K}_0 \rangle\!\rangle$) if necessary, we can express $g$ as
\begin{equation} \tag{$*$} \label{eq:g-expr}
g = g_1 z_1 X_{e_1} \cdots g_n z_n X_{e_n} g_{n+1} z_{n+1}
\end{equation}
where $e_1 \cdots e_n$ is a closed path in $\mathcal{G}$, $g_k \in \FF{i(e_k)}$ and $z_k \in \FFF{i(e_k)}$ for all $k$, with indices taken modulo $n$. In particular, we have an element $g \in \overline{\mathbb{F}}$ and an expression \eqref{eq:g-expr}, where $e_1 \cdots e_n$ is a closed path in $\mathcal{G}$, where $z_k \in \FFFF{i(e_k)}$, and where $g_k \in \widehat{\mathbb{F}}$ are such that $\overline\varphi_j(g_k) \in F_{i(e_k)}$ $\omega$-almost surely. We will prove, by induction on $n$, that such an element $g$ is contained in $\langle\!\langle \mathcal{K} \rangle\!\rangle$, as long as $\overline\varphi_j(g) = 1$ $\omega$-almost surely.

In the base case, $n = 0$, we have $\overline\varphi_j(g) = \overline\varphi_j(g_1) \overline\varphi_j(z_1)$ for all $j$, and we $\omega$-almost surely have $\overline\varphi_j(g_1) \in F_v$ and $\overline\varphi_j(z_1) \in Z_v$, for some $v \in V(\mathcal{G})$. Since $\overline\varphi_j(g) = 1$ $\omega$-almost surely and since $F_v \cap Z_v = \{1\}$, it follows that we have $\overline\varphi_j(g_1) = \overline\varphi_j(z_1) = 1$ $\omega$-almost surely and so $g_1,z_1 \in \mathcal{K}_1$. We therefore have $g \in \langle\!\langle \mathcal{K} \rangle\!\rangle$, as required.

Suppose now that $n > 0$. Then, by Theorem~\ref{thm:gog-normal-forms}, it follows that there $\omega$-almost surely exists $k \in \{2,\ldots,n\}$ such that $e_{k-1} = \overline{e_k}$ and $\overline\varphi_j(g_kz_k) \in \iota_{e_k}(G_{e_k})$; since $\omega$ is finitely additive, this holds $\omega$-almost surely for some fixed value of $k$. Let $e = e_{k-1}$, $v = i(\overline{e})$ and $w = i(e)$. If $z_k = s_{1,v} \cdots s_{\ell,v}$, we can then use the elements $s_{m,w} s_{m,e} X_e s_{m,v}^{-1} X_e^{-1} \in \mathcal{K}_2$ to obtain
\begin{align*}
X_{e_{k-1}} g_k z_k X_{e_k} &= X_e g_k z_k X_{\overline{e}} \equiv_0 X_e g_k z_k X_e^{-1} = X_eg_kX_e^{-1} \cdot \prod_{m=1}^\ell X_es_{m,v}X_e^{-1} \\ &\equiv_2 X_eg_kX_e^{-1} \cdot \prod_{m=1}^\ell s_{m,w}s_{m,e}
\end{align*}
where ``$\equiv_i$'' denotes equivalence modulo $\langle\!\langle \mathcal{K}_i \rangle\!\rangle$. By construction we then have $s_{m,w} \in \FFFF{w}$, whereas since $\overline\varphi_j(s_{m,v}) \in Z_v \subseteq \iota_{\overline{e}}(G_e)$ and $\overline\varphi_j(s_{m,w}) \in Z_w \subseteq \iota_e(G_e)$ by \ref{it:cyclic-He}, it follows that we have $\overline\varphi_j(s_{m,e}) = \overline\varphi_j(s_{m,w})^{-1} e \overline\varphi_j(s_{m,v}) e^{-1} \in F_w$ $\omega$-almost surely. It therefore follows that each $s_{m,w}$ commutes with each $s_{m',e}$ modulo $\langle\!\langle \mathcal{K}_3 \rangle\!\rangle$, and so we may replace $\prod_{m=1}^\ell s_{m,w} s_{m,e}$ with $g'z'$, where $g' = \prod_{m=1}^\ell s_{m,e}$ and $z' = \prod_{m=1}^\ell s_{m,w}$. Finally, since $\overline\varphi_j(g_{k+1}) \in F_w$ $\omega$-almost surely by construction and since $\overline\varphi_j(g_k) \in \iota_{\overline{e}}(G_e)$ (implying that $e \overline\varphi_j(g_k) e^{-1} \in F_w$) $\omega$-almost surely, the commutators of the elements $z_{k-1}$ and $z'$ with the elements $X_e g_k X_e^{-1}$, $g'$ and $g_{k+1}$ all belong to $\mathcal{K}_3$, which allows us to replace the subword $g_{k-1}z_{k-1}X_{e_{k-1}}g_kz_kX_{e_k}g_{k+1}z_{k+1}$ of $g$ with the product $(g_{k-1} \cdot X_eg_kX_e^{-1} \cdot g' \cdot g_{k+1}) \cdot (z_{k-1} \cdot z' \cdot z_{k+1})$. Such a replacement procedure results in an element satisfying the inductive hypothesis of normal form length $n-2$; it follows that such an element, and therefore $g$, is in $\langle\!\langle \mathcal{K} \rangle\!\rangle$, as required.
\end{proof}

\begin{proof}[Proof of Proposition~\ref{prop:cyclic-ext}]
Let $\mathcal{T}$ be the corresponding Bass--Serre tree. Assume, without loss of generality, that $G$ acts on $\mathcal{T}$ without global fixed points (otherwise $G = F_v \times Z_v \cong F_v \times (G/F)$ for some $v \in V(\mathcal{G})$, and so $G$ is equationally Noetherian as a product of two equationally Noetherian groups), and in particular with unbounded orbits.

Let $(\psi_j\colon \mathbb{F} \to G)_{j \in \mathbb{N}}$ be a non-divergent sequence of homomorphisms (with respect to the action on $\mathcal{T}$ and a non-principal ultrafilter $\omega$), for some finitely generated free group $\mathbb{F}$ with free basis $S$; it is then enough to show, by Theorem~\ref{thm:en-non-divergent}, that $\ker^\omega(\psi_j) \subseteq \ker(\psi_j)$ $\omega$-almost surely. Since the sequence $(\psi_j)$ is non-divergent, it follows that $\lim_{j \to \omega} \inf_{x \in \mathcal{T}} \max_{s \in S} d(x, \psi_j(s) \cdot x) < \infty$. Let $x_0 \in \mathcal{T}$ be a base vertex; since the $G$-action on $\mathcal{T}$ is cocompact, we may then assume, by conjugating each $\psi_j$ if necessary, that the infimum is $\omega$-almost surely realised uniformly close to $x_0$, implying that we have $n \coloneqq \lim_{j \to \omega} \max_{s \in S} d(x_0, \psi_j(s) \cdot x_0) < \infty$.

In particular, given $s \in S$ the geodesic path from $x_0$ to $\psi_j(s) \cdot x_0$ in $\mathcal{T}$ $\omega$-almost surely has length $\leq n$, so since the underlying graph of $\mathcal{G}$ is finite (and therefore contains finitely many paths of length $\leq n$) and since $\omega$ is finitely additive, it follows that there is a specific closed path $P_s = e_{s,1} \cdots e_{s,n_s}$ in $\mathcal{G}$ (with $n_s \leq n$) such that $P_s$ is $\omega$-almost surely the image of the geodesic from $x_0$ to $\psi_j(s) \cdot x_0$ under the quotient map. This implies that we $\omega$-almost surely have
\[
\psi_j(s) = h_{s,1,j} e_{s,1} \cdots h_{s,n_s,j} e_{s,n_s} h_{s,n_s+1,j},
\]
where each $h_{s,k,j} \in F_{i(e_{s,k})} \times Z_{i(e_{s,k})}$ (where $e_{s,n_s+1} \coloneqq e_{s,1}$), and in particular $h_{s,k,j} = g_{s,k,j} z_{s,k,j}$ for some $g_{s,k,j} \in F_{i(e_{s,k})}$ and $z_{s,k,j} \in Z_{i(e_{s,k})}$.

Now let $\FF{E}$ be a free group with basis $\{ X_e \mid e \in E(\mathcal{G}) \}$, and for each $v \in V(\mathcal{G})$ let $\FF{v}$ and $\FFF{v}$ be free groups with bases $\{ Y_{s,k} \mid s \in S, 1 \leq k \leq n_s+1, i(e_{s,k}) = v \}$ and $\{ \widehat{Y}_{s,k} \mid s \in S, 1 \leq k \leq n_s+1, i(e_{s,k}) = v \}$, respectively. Let $\widehat{\mathbb{F}} = \left( *_{v \in V(\mathcal{G})} (\FF{v} * \FFF{v}) \right) * \FF{E}$, and define homomorphisms $\chi\colon \mathbb{F} \to \widehat{\mathbb{F}}$ and $(\widehat\psi_j\colon \widehat{\mathbb{F}} \to G)_{j \in \mathbb{N}}$ by setting $\chi(s) = Y_{s,1} \widehat{Y}_{s,1} X_{e_{s,1}} \cdots Y_{s,n_s} \widehat{Y}_{s,n_s} X_{e_{s,n_s}} Y_{s,n_s+1} \widehat{Y}_{s,n_s+1}$ for $s \in S$, $\widehat\psi_j(X_e) = e$ for $e \in E(\mathcal{G})$, as well as $\widehat\psi_j(Y_{s,k}) = g_{s,k,j}$ and $\widehat\psi_j(\widehat{Y}_{s,k}) = z_{s,k,j}$ (which are well-defined $\omega$-almost surely).  It follows that we have $\psi_j = \widehat\psi_j \circ \chi$ whenever $\widehat\psi_j$ is defined, and in particular we have $\ker(\psi_j) = \chi^{-1}(\ker(\widehat\psi_j))$ $\omega$-almost surely, as well as $\ker^\omega(\psi_j) = \chi^{-1}(\ker^\omega(\widehat\psi_j))$. However, it follows from Lemma~\ref{lem:sorry} that we $\omega$-almost surely have $\ker^\omega(\widehat\psi_j) \subseteq \ker(\widehat\psi_j)$, so since $\chi$ is injective it follows that $\ker^\omega(\psi_j) \subseteq \ker(\psi_j)$ $\omega$-almost surely. Thus, by Theorem~\ref{thm:en-non-divergent}, $G$ is equationally Noetherian, as required.
\end{proof}

In order to prove Proposition~\ref{prop:en-linear}, it remains to show that we can apply Proposition~\ref{prop:cyclic-ext} to the setting of free-by-cyclic groups with linearly growing monodromies. In order to satisfy the assumptions \ref{it:cyclic-Zv}--\ref{it:cyclic-acyl} of Proposition~\ref{prop:cyclic-ext}, we will use the trees constructed by Andrew--Martino \cite{AndrewMartino2022}. We note that there exists a related construction of a tree with a non-elementary 4-acylindrical action which is due to Dahmani--Touikan \cite{DahmaniTouikan2023}.

We remark that in general the trees constructed in \cite{AndrewMartino2022} differ from the trees arising in the work of Guirardel--Horbez in Theorem~\ref{thm:GH}.

\begin{lem}[{Andrew--Martino \cite[Proposition~5.2.2]{AndrewMartino2022}}]\label{lem:linear-growth-splitting}
    Let $F$ be a free group of finite rank and let $\varphi \in \Aut(F)$ be an automorphism representing a unipotent and linearly growing element $\Phi$ in $\Out(F)$. Let $G = F \rtimes_{\varphi} \langle t \rangle$ be the corresponding mapping torus. Then there exists a bipartite $G$-tree $(\mathcal{T}, V_0(\mathcal{T}), V_1(\mathcal{T}))$ with the following properties.
    \begin{enumerate}
        \item \label{it:linear-black-vertex-stabs} The stabiliser of each vertex $u \in V_0(\mathcal{T})$ is a subgroup of the form $F_u \times \langle t_u \rangle$, where $F_u \leq F$ is a maximal infinite cyclic subgroup of $F$, and $t_u \in Ft$.
        \item \label{it:linear-vertex-stabs} The stabiliser of each vertex $v \in V_1(\mathcal{T})$ is a maximal subgroup of the form $F_v \times \langle t_v \rangle$ where $F_v \leq F$ is a finitely generated non-abelian subgroup of $F$ and $t_v \in Ft$. Moreover, if $v_0$ and $v_1$ are distinct vertices in $V_1(\mathcal{T})$ then $t_{v_0} \neq t_{v_1}$.
        \item \label{it:linear-edge-stabs} The stabiliser of every edge is a maximal $\ZZ^2$-subgroup of $G$.
        \item \label{it:linear-center} For every edge $e$ in $\mathcal{T}$ and incident vertex $v$, we have that $Z(\Stab_G(v)) \leq \Stab_G(e)$, where $Z(\Stab_G(v))$ denotes the center of $\Stab_G(v)$.
        \item \label{it:linear-acylindrical} The action of $G$ on $\mathcal{T}$ is $4$-acylindrical.
    \end{enumerate}
\end{lem}

\begin{proof}
    The construction in Proposition~5.2.2 of \cite{AndrewMartino2022} gives a $G$-tree $\mathcal{T}$ which satisfies Items~\ref{it:linear-black-vertex-stabs}--\ref{it:linear-edge-stabs}.
    
    For Item~\ref{it:linear-center}, we note that for each edge $e$ in $\mathcal{T}$  and incident vertex $v$, the edge stabiliser $\Stab_G(e)$ embeds as a maximal abelian subgroup of $\Stab_G(v)$ and thus contains its center.

    For Item~\ref{it:linear-acylindrical}, we observe that Items~\ref{it:linear-black-vertex-stabs} and \ref{it:linear-vertex-stabs} imply that the $G$-tree $\mathcal{T}$ is $\Phi$-aligned. Indeed, for any two vertices $u_1, u_2 \in V_0(\mathcal{T})$, we have that $\Stab_G(u_i) \cap F = F_{u_i}$ for $i = 1,2$. Since each $F_{u_i}$ is a maximal cyclic subgroup of $F$, if the intersection $F_{u_1} \cap F_{u_2}$ is non-trivial then $F_{u_1} = F_{u_2}$. Since $t_{u_2}$ centralises $F_{u_2} = F_{u_1}$, it follows that $t_{u_2} \in F_{u_1} \times \langle t_{u_1} \rangle$ and $t_{u_2} = xt_{u_1}$, for some $x \in F_{u_1}$. Hence,
    \[\langle F_{u_2}, t_{u_2}\rangle = \langle F_{u_1}, xt_{u_1} \rangle = \langle F_{u_1}, t_{u_1} \rangle.\] However, by the construction of the tree $\mathcal{T}$ in \cite[Proposition~5.2.2]{AndrewMartino2022}, for any two vertices $u_1, u_2 \in V_0(\mathcal{T})$, if $\Stab_G(u_1) = \Stab_G(u_2)$ then $u_1 = u_2$. Hence Item~\ref{it:aligned_splitting_ffs} of Definition~\ref{defn:aligned_splitting} is satisfied.

    For Item~\ref{it:aligned_splittings_black}, we define the map
    \[\begin{split}
        V_1(\mathcal{T}) &\mapsto \pi^{-1}(\Phi) \\
        v &\mapsto \varphi_v
    \end{split}\]
    so that $\varphi_v \in \Aut(F)$ is the automorphism of $F$ induced by the conjugation action of $t_v$ on $F$ in $G$. If $u \in V_0(\mathcal{T})$ is adjacent to $v$ in $\mathcal{T}$, then the element $t_v$ is contained in the stabiliser of $u$, since $\Stab_G(u)$ embeds as a maximal $\ZZ^2$-subgroup of $\Stab_G(v) = F_v \times \langle t_v \rangle$. The map is clearly injective.
    
    Moreover, since the outer automorphism $\Phi$ is UPG, it is neat and thus by Lemma~\ref{lemma:non-periodic_trees}, the action on $\mathcal{T}$ is aperiodic. Thus by Proposition~\ref{prop:aligned_acylindrical_trees}, the action is 4-acylindrical.
\end{proof}

\begin{remark}\label{rmk:edge-elements}
    By Lemma~\ref{lem:linear-growth-splitting}, every edge of the tree $\mathcal{T}$ is stabilised by an element of the coset $Ft$, and therefore by an element of every coset of $F$ in $G$. Hence, the quotient graph $G \backslash \mathcal{T}$ is isomorphic to the quotient $F \backslash \mathcal{T}$ where the action of $F$ on $\mathcal{T}$ is obtained by restricting the $G$-action. It follows that in the induced graph-of-groups splitting of the mapping torus $G$ we may always take the edge elements to be contained in the subgroup $F$.
\end{remark}

Finally, in order to show that the splitting described above satisfies assumption~\ref{it:cyclic-qalg} of Proposition~\ref{prop:cyclic-ext}, we will use the following result, stating that quasi-algebraicity of subgroups passes to finite extensions of groups. Our proof closely follows the proof of its special case when $K = \{1\}$ \cite[Theorem~1]{BaumslagMyasnikovRomankov1997}.

\begin{prop} \label{prop:quasi-alg-extensions}
Let $G$ be a group, let $H \leq G$ be a finite-index subgroup, and let $K$ be a quasi-algebraic subgroup of $H$. Then $K$ is quasi-algebraic in $G$.
\end{prop}

\begin{proof}
We first reduce to the case when $G = H \rtimes Q$ for a finite group $Q$. We do this in two steps: reducing to the case when $H$ is normal in $G$, and then reducing to the case when the extension of $H$ by $G/H$ is split.

Let $N = \bigcap_{g \in G} g^{-1}Hg$, so that $N$ is a finite-index subgroup of $H$. Since $K$ is quasi-algebraic in $H$, it then follows that $K \cap N$ is quasi-algebraic in $N$, since $V_{N,K \cap N}(S) = V_{H,K}(S) \cap N^m$ for all $S \subseteq \mathbb{F}_m$. On the other hand, $K \cap N$ is a finite-index subgroup of $K$, and finite extensions of quasi-algebraic subgroups are quasi-algebraic (this follows since translates and finite unions of quasi-algebraic subsets are quasi-algebraic \cite[Proposition~7.5]{Valiunas2021}); it is thus enough to show that $K \cap N$ is quasi-algebraic in $G$. We may therefore replace $H$ and $K$ with $N$ and $K \cap N$, respectively, to assume that $H$ is normal in $G$.

Now let $Q = G/H$, and note that there is an injective homomorphism $\phi\colon G \to H \wr Q$ sending any element $h \in H$ to an element $(h_q \mid q \in Q) \in H^Q \leq H \wr Q$ such that $h_1 = h$, and any element of $G \setminus H$ to an element outside $H^Q$: see \cite[Theorem~22.21]{Neumann1967}. In particular, we have $K = \phi^{-1}(K')$, where $K' = \pi^{-1}(K)$ for the projection $\pi\colon H^Q \to H$ sending $(h_q \mid q \in Q) \mapsto h_1$, so since $\phi$ is injective it is enough to show that $K'$ is quasi-algebraic in $H \wr Q$. Note, however, since $K$ is quasi-algebraic in $H$ it follows that $K'$ is quasi-algebraic in $H^Q$ (we may use the same finite subsequence of equations in both cases to show this). We may therefore replace $G$, $H$ and $K$ with $H \wr Q$, $H^Q$ and $K'$, respectively, to assume that the short exact sequence $1 \to H \to G \to Q \to 1$ is split, i.e.\ $G = H \rtimes Q$.

We now claim that given any system $S \subseteq \mathbb{F}_m$ of positive equations (i.e.\ words only containing positive powers of the basis elements of $\mathbb{F}_m$), there exists a finite subsystem $S_0 \subseteq S$ such that $V_{G,K}(S_0) = V_{G,K}(S)$. Such a result will imply that $K$ is quasi-algebraic in $G$: see \cite[Lemma~3.1]{Valiunas2021} and its proof.

Given a positive equation $s = X_{j_1} \cdots X_{j_n} \in \mathbb{F}_m(X_1,\ldots,X_m)$ and elements $h_1,\ldots,h_m \in H$ and $q_1,\ldots,q_m \in Q$, note that we can write
\[
s(h_1q_1,\ldots,h_mq_m) = s(q_1,\ldots,q_m) \cdot \prod_{i=1}^n h_{j_i}^{q_i'},
\]
where $q_i' = q_{j_i} \cdots q_{j_n}$. The first and the second terms in this expression are elements of $Q$ and $H$, respectively, implying that $s(h_1q_1,\ldots,h_mq_m) \in K$ if and only if the first of these terms is equal to the identity and the second one is in $K$. Note, moreover, that the tuple $(q_1',\ldots,q_n')$ depends only on the tuple $(q_1,\ldots,q_m)$; therefore, if we fix $\mathbf{q} = (q_1,\ldots,q_m) \in Q^m$ then given $(h_1,\ldots,h_m) \in H^m$ we have $\prod_{i=1}^n h_{j_i}^{q_i'} \in K$ if and only if $s'_{\mathbf{q}}(h_j^q \mid 1 \leq j \leq m, q \in Q) \in K$, where we define the equation $s'_{\mathbf{q}} \coloneqq \prod_{i=1}^n Y_{j_i}^{(q_i')} \in \mathbb{F}_{m|Q|}(Y_j^{(q)} \mid 1 \leq j \leq m, q \in Q)$.

Now given any collection of positive equations $S \subseteq \mathbb{F}_m$ and a tuple $\mathbf{q} = (q_1,\ldots,q_m) \in Q^m$, let $S[\mathbf{q}] \coloneqq \{ s'_{\mathbf{q}} \mid s \in S \}$. Fix such a collection $S \subseteq \mathbb{F}_m$. Since $Q$ is finite, and therefore equationally Noetherian, there exists a finite subset $S_0 \subseteq S$ such that $V_Q(S_0) = V_Q(S)$. Moreover, since $K$ is quasi-algebraic in $H$, for each $\mathbf{q} \in Q^m$ there exists a finite subset $S_{\mathbf{q}} \subseteq S$ such that $V_{H,K}(S_{\mathbf{q}}[\mathbf{q}]) = V_{H,K}(S[\mathbf{q}])$. Now let $S' = S_0 \cup \bigcup_{\mathbf{q} \in Q^m} S_{\mathbf{q}}$, and note that since $Q$ is finite so is $S'$. Then we have $V_Q(S') = V_Q(S)$ and $V_{H,K}(S'[\mathbf{q}]) = V_{H,K}(S[\mathbf{q}])$ for all $\mathbf{q} \in Q^m$, so it follows from the above argument that $V_{G,K}(S') = V_{G,K}(S)$. This implies that $K$ is quasi-algebraic in $G$, as required.
\end{proof}

\begin{proof}[Proof of Proposition~\ref{prop:en-linear}]
By Lemma~\ref{lem:linear-growth-splitting}, the group $G$ (and its subgroup $F$) admits a splitting satisfying assumptions \ref{it:cyclic-Zv}, \ref{it:cyclic-He} and \ref{it:cyclic-acyl} of Proposition~\ref{prop:cyclic-ext}, and by Remark~\ref{rmk:edge-elements} it also satisfies assumption \ref{it:cyclic-edges-in-G}. Moreover, each group $F_v$ (for $v \in V(\mathcal{G})$) appearing in the statement of Proposition~\ref{prop:cyclic-ext} is a finitely generated subgroup of the finitely generated free group $F$, so by Marshall Hall's Theorem \cite[Proposition~I.3.10]{LyndonSchupp1977} it is a free factor, and therefore a retract, of a finite-index subgroup $H_v \leq F$ containing $F_v$. Since the group $H_v$ is finitely generated free, and therefore equationally Noetherian, it follows from \cite[Proposition~7.5(v)]{Valiunas2021} that $F_v$ is quasi-algebraic in $H_v$. Therefore, by Proposition~\ref{prop:quasi-alg-extensions}, $F_v$ is quasi-algebraic in $F$, and so assumption~\ref{it:cyclic-qalg} of Proposition~\ref{prop:cyclic-ext} is satisfied. Thus $G$ is equationally Noetherian by Proposition~\ref{prop:cyclic-ext}.
\end{proof}

\section{The general case}
\label{s:superlinear}

In this section we prove Theorem~\ref{thm:main}. Using relative hyperbolicity of free-by-cyclic groups with exponentially growing monodromy \cite{DahmaniLi2022,Ghosh2023} and the fact that being equationally Noetherian is preserved under relative hyperbolicity \cite[Theorem~D]{GrovesHull2019}, Theorem~\ref{thm:main} will follow from the following special case thereof.

\begin{prop}\label{prop:polynomially-growing-EN}
Let $F$ be a finitely generated free group, and let $\Phi \in \Out(F)$ be a polynomially growing outer automorphism. Then $G = F \rtimes_\Phi \ZZ$ is equationally Noetherian.
\end{prop}

We prove Proposition~\ref{prop:polynomially-growing-EN} by induction on the growth rate of the monodromy. The key tool is the existence of graph-of-groups splittings where the edge groups are infinite cyclic and the vertex groups are free-by-cyclic with monodromies which grow polynomially and of strictly lower degree. It has been observed by multiple authors (e.g. see \cite{Hagen2019} and \cite{AndrewHughesKudlinska2023}) that the existence of unipotent representatives for UPG outer automorphisms, combined with the work of Macura \cite{Macura2002}, gives rise to such splittings for free-by-cyclic groups with UPG monodromy which grows at least quadratically. We will moreover show that the action on the Bass--Serre tree corresponding to the splitting is 2-acylindrical. 

\begin{lem} \label{lem:polynomial-growth-splitting}
    Let $F$ be a free group of finite rank and let $\varphi \in \Aut(F)$ be an automorphism representing a UPG element of $\Out(F)$ with polynomial growth of degree $d \geq 2$. Let $G = F \rtimes_{\varphi} \langle t \rangle$ be the corresponding mapping torus. Then there exists a $G$-tree $\mathcal{T}$ with the following properties. 
    \begin{enumerate}
        \item\label{it:splitting-1} Each vertex stabiliser is of the form $F_v \rtimes_{\psi} \langle t_v \rangle$ where $F_v \leq F$ is a finitely generated subgroup of $F$, $t_v \in F t$ and the automorphism $\psi \in \Aut(F_v)$ represents a unipotent automorphism with polynomial growth of order $d_v$ where $1 \leq d_v < d$.
        \item\label{it:splitting-2} Each edge stabiliser is of the form $\langle t_e \rangle$ where $t_e \in Ft$.
        \item\label{it:splitting-3} The action of $G$ on $\mathcal{T}$ is $2$-acylindrical.
    \end{enumerate}
\end{lem}

\begin{proof}
The splitting constructed in \cite{AndrewHughesKudlinska2023} satisfies conditions \ref{it:splitting-1} and \ref{it:splitting-2}. Hence it remains to show that condition \ref{it:splitting-3} is satisfied. We use the following explicit description of the splitting described in \cite[Proposition~2.5]{AndrewHughesKudlinska2023}. Let $\Phi \in \Out(F)$ be a UPG outer automorphism with growth of degree $d \geq 2$, and let $f\colon \Gamma \to \Gamma$ be a unipotent representative for $\Phi$ (referred to as the ``improved relative train track'' in \cite[\S2]{AndrewHughesKudlinska2023}). Let $e_1,\ldots,e_m$ be the edges of $\Gamma$ whose growth has degree $d$, and let $\Gamma_1,\ldots,\Gamma_n$ be the connected components obtained by removing from $\Gamma$ the interiors of the edges $e_i$. Then $f$ preserves each $\Gamma_j$, and we have $f(e_k) = e_k \rho_k$ for all $k$, where $\rho_k$ is a closed path in the connected component $\Gamma_{j_k}$ that contains $i(\overline{e_k})$. Moreover, since each $e_k$ has growth of order $d \geq 2$, we have that $[f(\rho_k)] \neq \rho_k$.

The splitting graph of groups $\mathcal{G}$ for $G = F \rtimes_\Phi \ZZ$ can then be described as follows. Let the underlying graph $|\mathcal{G}|$ of $\mathcal{G}$ be the graph obtained from $\Gamma$ by collapsing each $\Gamma_j$ to a vertex $v_j$. The map $f$ then induces the identity map on $|\mathcal{G}|$, and therefore lifts to an automorphism $f'\colon \mathcal{T} \to \mathcal{T}$, where $\mathcal{T}$ is the Bass--Serre tree corresponding to $\mathcal{G}$. Blowing up every lift of $v_j$ in $\mathcal{T}$ to a tree isomorphic to $\widetilde\Gamma_j$ we obtain the universal cover $\widetilde\Gamma$ of $\Gamma$, and the map $f$ lifts to a quasi-isometry $\widetilde{f}\colon \widetilde\Gamma \to \widetilde\Gamma$ that induces the map $f'$ when the subtrees corresponding to translates of $\widetilde\Gamma_j$ are collapsed. In particular, the maps $f$, $f'$ and $\widetilde{f}$ fit into the diagram
\[
\begin{tikzcd}
\widetilde\Gamma \ar[rrr, "\widetilde{f}", dashed] \ar[ddd, "q"', two heads] \ar[dr, two heads] &&& \widetilde\Gamma \ar[ddd, "q", two heads, pos=0.6] \ar[dr, two heads] \\
& \Gamma \ar[rrr, "f", dashed, crossing over, pos=0.4] &&& \Gamma \ar[ddd, two heads] \\ \\
\mathcal{T} \ar[rrr, "f'", pos=0.6] \ar[dr, two heads] &&& \mathcal{T} \ar[dr, two heads] \\
& \left|\mathcal{G}\right| \ar[rrr, equal] \ar[from=uuu, two heads, crossing over] &&& \left|\mathcal{G}\right|
\end{tikzcd}
\]
where the dashed arrows are maps that are not necessarily simplicial and the vertical arrows correspond to collapsing the subgraphs $\Gamma_j$ or their lifts, and where the front face commutes up to homotopy, the back face up to $F$-equivariant quasi-isometry, and the other four faces commute strictly.

Now by the construction, the edge stabilisers of the $G$-action on $\mathcal{T}$ are cyclic, generated by appropriate maps $f'\colon \mathcal{T} \to \mathcal{T}$ as above. On the other hand, if a lift $\widetilde{e}_k \in E(\widetilde\Gamma)$ of the edge $e_k \in E(\Gamma)$ has its image $q(\widetilde{e}_k) \in E(\mathcal{T})$ stabilised by such an $f'$, it follows that we have $\widetilde{f}(\widetilde{e}_k) = \widetilde{e}_k \widetilde{\rho}_k$, where $\widetilde{\rho}_k$ is a path in an appropriate lift $\widetilde\Gamma_j$ and therefore $q(\widetilde{\rho}_k)$ is a single vertex of $\mathcal{T}$.

We claim that if a non-trivial element $g \in G$ stabilises a lift $e_k' \in E(\mathcal{T})$ of the edge $e_k$, then it cannot stabilise any other edge of $\mathcal{T}$ incident to $i(\overline{e_k'})$. Indeed, suppose for contradiction that $g$ stabilises a $2$-path $e_k' e'$ in $\mathcal{T}$, where $e'$ is a lift of some $e \in \{ e_\ell,\overline{e_\ell} \}$. Then $g = (f')^r$ for some $r \neq 0$ (without loss of generality, $r > 0$) and a map $f'$ as above. Letting $\widetilde{e}_k$ and $\widetilde{e}$ be the unique edges of $\widetilde\Gamma$ such that $q(\widetilde{e}_k) = e_k'$ and $q(\widetilde{e}) = e'$, it follows that $\widetilde{f}^r(\widetilde{e}_k) = \widetilde{e}_k \widetilde{\rho}_k$, and either $\widetilde{f}^r(\widetilde{e}) = \widetilde{e} \widetilde{\rho}_\ell$ (if $e = e_\ell$) or $\widetilde{f}^r(\widetilde{e}) = \overline{\widetilde{\rho}_\ell} \widetilde{e}$ (if $e = \overline{e_\ell}$), where the paths $\widetilde{\rho}_k$ and $\widetilde{\rho}_\ell$ have trivial images under $q$. In particular, $\widetilde{f}^r$ fixes the vertices $v,w \in \widetilde\Gamma$, where $v = i(\widetilde{e}_k)$ and either $w = i(\widetilde{e})$ (if $e = e_l$) or $w = i(\overline{\widetilde{e}})$ (if $e = \overline{e_l}$), implying that $[\widetilde{f}^r(\widetilde{\sigma})] = \widetilde{\sigma}$, where $\widetilde{\sigma}$ is the geodesic path in $\widetilde\Gamma$ from $v$ to $w$.

Now it follows that $[f^r(\sigma)] = \sigma$, and therefore $[f^{pr}(\sigma)] = \sigma$ for all $p \geq 1$, where $\sigma$ is the image of $\widetilde{\sigma}$ in $\Gamma$ under the covering map. Moreover, by construction we have either $\sigma = e_k \sigma'$ (if $e = e_\ell$) or $\sigma = e_k \sigma' \overline{e_\ell}$ (if $e = \overline{e_\ell}$), where $\sigma'$ is a path contained in some component $\Gamma_j$ (the one containing $i(\overline{e_k})$); in the latter case, since $\widetilde{e}_k \neq \overline{\widetilde{e}}$, we have either $\ell \neq k$ or $\sigma'$ non-trivial. In particular, it follows that the path $\sigma$ crosses $e_k$. But then by Proposition~\ref{prop:eigen-rays}, given any initial segment $\tau$ of the infinite ray $e_k \rho_k [f(\rho_k)] [f^2(\rho_k)] \cdots$, one of $\tau$ and $\overline{\tau}$ is contained in $[f^q(\sigma)]$ for all $q$ large enough. If we take $\tau$ to be longer than $\sigma$, this contradicts the fact that $[f^{pr}(\sigma)] = \sigma$ for all $p \geq 1$. This proves the claim: that is, if a non-trivial element of $G$ stabilises a lift $e_k' \in E(\mathcal{T})$ of $e_k$ then it does not stabilise any other edge of $\mathcal{T}$ incident to $i(\overline{e_k'})$.

The claim now implies that if a non-trivial element $g \in G$ stabilises a lift $e_k' \in E(\mathcal{T})$ of $e_k$, then any other edge stabilised by $g$ is incident to $i(e_k')$. Consequently, if $v,w \in \mathcal{T}$ are two points distance $> 2$ apart, then no non-trivial element of $G$ stabilises both $v$ and $w$. In particular, the $G$-action on $\mathcal{T}$ is $2$-acylindrical, as required.
\end{proof}

\begin{proof}[Proof of Proposition~\ref{prop:polynomially-growing-EN}]

We begin by proving the result in the case that $\Phi$ is UPG. We prove the statement by induction on the degree of growth of $\Phi$. If $\Phi$ has constant growth then it has finite order in $\Out(F)$ by Lemma~\ref{lemma:slow-growth}, implying that $G$ contains $F \times \ZZ$ as a finite-index subgroup, so since $F \times \ZZ$ is equationally Noetherian so is $G$ \cite[Theorem~1]{BaumslagMyasnikovRomankov1997}. If $\Phi$ has linear growth, then the result follows from Proposition~\ref{prop:en-linear}.

Suppose that $\Phi$ has growth of degree $d \geq 2$, and let $\mathcal{G}$ be the graph of groups for the splitting of $G$ described in Lemma~\ref{lem:polynomial-growth-splitting}, and let $G_v = F_v \rtimes \langle t_v \rangle$ be the vertex group corresponding to $v \in V(\mathcal{G})$. We aim to use \cite[Theorem~1.9]{Valiunas2021} to show that $G$ is equationally Noetherian. Indeed, for any edge $e \in E(\mathcal{G})$ there exists a homomorphism $r_e\colon G_{i(e)} \to \langle t_e \rangle$ sending $F_{i(e)}$ to $1$ and $t_{i(e)}$ to $t_e$. We then have a map $\phi_e = \iota_{\overline{e}} \circ r_e\colon G_{i(e)} \to G_{i(\overline{e})}$ extending the isomorphism between the copies of $\langle t_e \rangle$ in the two vertex groups. Furthermore, since $t_v,t_e \in Ft$ for all $v \in V(\mathcal{G})$ and $e \in E(\mathcal{G})$ and since $\ker(r_e) = F_{i(e)}$ for all $e \in E(\mathcal{G})$, it follows that for any non-trivial closed path $e_1 \cdots e_n$ in $\mathcal{G}$ the composite $\phi_{e_n} \circ \cdots \circ \phi_{e_1}$ is equal to $\iota_{e_1} \circ r_{e_1}$. There are therefore only finitely many such composite homomorphisms, and thus the assumptions (i) and (ii) in \cite[Theorem~1.9]{Valiunas2021} are satisfied.

Now the action of $G$ on the corresponding Bass--Serre tree is acylindrical by Lemma~\ref{lem:polynomial-growth-splitting}\ref{it:splitting-3}. Furthermore, each of the groups $G_v$ is a free-by-cyclic group with UPG monodromy with growth of degree $< d$, so is equationally Noetherian by the inductive hypothesis. Since the maps $\iota_e \circ r_e$ are retractions of $G_{i(e)}$ onto $\iota_e(\langle t_e \rangle) \leq G_{i(e)}$, this also implies that $\iota_e(\langle t_e \rangle)$ is quasi-algebraic in $G_{i(e)}$: see \cite[Proposition~7.5(v)]{Valiunas2021}. It therefore follows from \cite[Theorem~1.9]{Valiunas2021} that $G$ is equationally Noetherian, as required.

For the general case, note that by Lemma~\ref{lem:upg-finite-index} there exists a positive integer $k$ such that $\Phi^k$ is UPG. Moreover, the mapping torus $F \rtimes_{\Phi^k} \ZZ$ can be identified with a finite-index subgroup of $G$. Hence by \cite[Theorem~1]{BaumslagMyasnikovRomankov1997}, the group $G$ is equationally Noetherian.
\end{proof}

\begin{proof}[Proof of Theorem~\ref{thm:main}]

If $\Phi$ is polynomially growing then by Proposition~\ref{prop:polynomially-growing-EN} the mapping torus $F \rtimes_{\Phi} \ZZ$ is equationally Noetherian. Otherwise, $\Phi$ is exponentially growing by the existence of relative train tracks \cite{BestvinaHandel1992}. Then by the work of Dahmani--Li \cite{DahmaniLi2022}, and independently Ghosh \cite{Ghosh2023}, the group $F \rtimes_{\Phi} \ZZ$ is hyperbolic relative to a finite collection of subgroups $\mathcal{P}$ where each $P_i \in \mathcal{P}$ is a free-by-cyclic group with polynomially growing monodromy. Then $F \rtimes_{\Phi} \ZZ$ is equationally Noetherian by Proposition~\ref{prop:polynomially-growing-EN} and \cite[Theorem~D]{GrovesHull2019}.
\end{proof}

\section{Well-ordering of growth rates}
\label{s:growth-rates}

Here we prove Corollary~\ref{cor:growth-rates-order}. Recall from the introduction that given a finitely generated group $G$, we denote by $\xi(G)$ the set of exponential growth rates of $G$. Recall, moreover, that a finitely generated subgroup $H = \langle T \rangle \leq G = \langle S \rangle$, where $|S|,|T| < \infty$, is said to be \emph{undistorted} if the map $\Cay(H,T) \to \Cay(G,S)$ induced by the inclusion $H \hookrightarrow G$ is a quasi-isometric embedding; this does not depend on the choices of $S$ and $T$.

\begin{lem} \label{lem:growth-rates}
Let $G$ be a finitely generated group, and let $Z \unlhd G$ be a finitely generated undistorted normal subgroup of sub-exponential growth. Then $\xi(G/Z) = \xi(G)$.
\end{lem}

\begin{proof}
Let $H = G/Z$, and given any $S \subseteq G$ we write $\overline{S} \coloneqq \{ sZ \mid s \in S \} \subseteq H$. It is clear that if $S$ is a finite generating set for $G$, then $\overline{S}$ is a finite generating set for $H$.

We claim that $e(G,S) = e(H,\overline{S})$ for all finite generating sets $S$ of $G$. Indeed, for any $n \in \mathbb{N}$ the quotient map $G \to H$ induces a surjective function $B_n(G,S) \to B_n(H,\overline{S})$, showing that $|B_n(G,S)| \geq |B_n(H,\overline{S})|$ and thus $e(G,S) \geq e(H,\overline{S})$. Conversely, since $Z$ is undistorted and of sub-exponential growth, there exists a sub-exponentially growing function $f\colon \mathbb{N} \to \mathbb{N}$ (i.e.\ a function satisfying $\lim_{n \to \infty} f(n)^{\frac{1}{n}} = 1$) such that $|B_n(G,S) \cap Z| \leq f(n)$ for all $n$. Now for any element $g_0 \in B_n(G,S)$ we have $\{ g^{-1}g_0 \mid g \in B_n(G,S) \cap g_0Z \} \subseteq B_{2n}(G,S) \cap Z$, implying that $|B_n(G,S) \cap g_0Z| \leq f(2n)$. This shows that the preimages of points under the map $B_n(G,S) \to B_n(H,\overline{S})$ have cardinality at most $f(2n)$ and therefore $|B_n(G,S)| \leq f(2n)|B_n(H,\overline{S})|$, so since $f$ grows sub-exponentially it follows that $e(G,S) \leq e(H,\overline{S})$. Hence $e(G,S) = e(H,\overline{S})$, as claimed.

Now given any $\lambda \in \xi(G)$, we have $\lambda = e(G,S)$ for some finite generating set $S$ of $G$, and therefore $\lambda = e(H,\overline{S}) \in \xi(H)$. Conversely, if $\lambda \in \xi(H)$ then $\lambda = e(H,T)$ for some finite generating set $T = \{ s_1Z,\ldots,s_kZ \}$ of $H$; if $S = \{ s_1,\ldots,s_k,z_1,\ldots,z_m \}$, where $\{ z_1,\ldots,z_m \}$ is a finite generating set for $Z$, it then follows that $S$ is a finite generating set of $G$ and $\overline{S} = T \cup \{1_H\}$, implying that $\lambda = e(H,\overline{S}) = e(G,S) \in \xi(G)$. Thus $\xi(G) = \xi(H)$, as required.
\end{proof}

\begin{proof}[Proof of Corollary~\ref{cor:growth-rates-order}]
  We may assume, without loss of generality, that $F$ is non-abelian (as otherwise $G$ is virtually abelian, so $\xi(G) = \varnothing$ is well-ordered).
  
  If $\Phi$ has exponential growth, then by the work of Dahmani--Li \cite{DahmaniLi2022} and Ghosh \cite{Ghosh2023}, the group $G$ is hyperbolic relative to a finite collection $\mathcal{P}$ of free-by-cyclic proper subgroups. By Theorem~\ref{thm:main}, each element of $\mathcal{P}$ is equationally Noetherian. Hence by Theorem~1.2 in \cite{Fujiwara2023}, the set $\xi(G)$ is well ordered. 
  
  Suppose now that $\Phi$ is polynomially growing and has infinite order. By Proposition~\ref{prop:acylindrical_actions}, the group $G$ admits a non-elementary and acylindrical action on a simplicial tree $\mathcal{T}$. If $S$ is a finite generating set for $G$ then $S$ or $S^2$ has an element that acts hyperbolically on $\mathcal{T}$ \cite[Corollary~2 on page~64]{Serre1980}. Moreover, $G$ is equationally Noetherian by Theorem~\ref{thm:main}. Hence the conditions of \cite[Theorem~1.1]{Fujiwara2023} are satisfied when $M = 2$, and we may conclude that $\xi(G)$ is well ordered. 

  Finally, suppose that $\Phi$ has finite order. Then the group $G$ has non-trivial center and so has a central element $z$ such that $\psi(z) \neq 0$, where $\psi\colon F \rtimes_\Phi \ZZ \to \ZZ$ is the group homomorphism defined by $\psi(F) = 0$ and $\psi|_{\ZZ} = \Id_{\ZZ}$. In particular, $Z = \langle z \rangle$ is a central (and therefore normal) subgroup of $G$ that is cyclic (and therefore of sub-exponential growth) and undistorted in $G$ (since $\psi$ maps it injectively to a finite-index subgroup of $\ZZ$). It then follows from Lemma~\ref{lem:growth-rates} that $\xi(G) = \xi(G/Z)$. On the other hand, the group $G/Z$ is an extension of the free group $F$ by the finite group $\ZZ/\psi(Z)$, implying that $G/Z$ is virtually free and therefore Gromov hyperbolic. It then follows from Theorem~2.2 in \cite{FujiwaraSela2023} that $\xi(G) = \xi(G/Z)$ is well-ordered, as required.
\end{proof}

\appendix

\bibliographystyle{amsalpha}
\bibliography{ref}

\end{document}